\documentclass{proc-l}

\usepackage{amssymb,amsmath,amsthm,graphicx}

\theoremstyle{plain}
\newtheorem{theorem}{Theorem}[section]
\newtheorem{corollary}[theorem]{Corollary}
\newtheorem{lemma}[theorem]{Lemma}
\newtheorem{proposition}[theorem]{Proposition}
 
\theoremstyle{definition}

\theoremstyle{remark}
\newtheorem*{remark}{Remark}

\numberwithin{equation}{section}

\newcommand{\CC}{{\mathbb C}}
\newcommand{\DD}{{\mathbb D}}
\newcommand{\RR}{{\mathbb R}}
\newcommand{\TT}{{\mathbb T}}

\let\Re\undefined
\DeclareMathOperator{\Re}{Re}
\DeclareMathOperator{\td}{teardrop}
\DeclareMathOperator{\conv}{conv}

\begin{document}

\title{On mapping theorems for numerical range}

\author[Klaja]{Hubert Klaja}
\address{D\'epartement de math\'ematiques et de statistique, Universit\'e Laval, 
1045 avenue de la M\'edecine, Qu\'ebec (QC), Canada G1V 0A6}
\email{hubert.klaja@gmail.com}

\author[Mashreghi]{Javad Mashreghi}
\address{D\'epartement de math\'ematiques et de statistique, Universit\'e Laval, 
1045 avenue de la M\'edecine, Qu\'ebec (QC), Canada G1V 0A6}
\email{javad.mashreghi@mat.ulaval.ca}
\thanks{Second author supported by NSERC}

\author[Ransford]{Thomas Ransford}
\address{D\'epartement de math\'ematiques et de statistique, Universit\'e Laval, 
1045 avenue de la M\'edecine, Qu\'ebec (QC), Canada G1V 0A6}
\email{ransford@mat.ulaval.ca}
\thanks{Third author  supported by NSERC and the Canada research chairs program.}


\date{April 24, 2015}

\commby{??}

\begin{abstract}
Let $T$ be an operator on a Hilbert space $H$ with numerical radius  $w(T)\le1$. 
According to a theorem of Berger and Stampfli, 
if $f$ is a function in the disk algebra such that $f(0)=0$, 
then $w(f(T))\le\|f\|_\infty$. 
We give a new and elementary proof of this result using finite Blaschke products.

A well-known result relating numerical radius and norm says $\|T\| \leq 2w(T)$. 
We obtain a local  improvement of this estimate, namely, if $w(T)\le1$ then
\[
\|Tx\|^2\le 2+2\sqrt{1-|\langle Tx,x\rangle|^2}
\qquad(x\in H,~\|x\|\le1).
\]
Using this refinement, 
we give a simplified proof of Drury's teardrop theorem, 
which extends the Berger--Stampfli theorem to the case $f(0)\ne0$.
\end{abstract}

\maketitle

\section{Introduction}

Let $H$ be a complex Hilbert space
and  $T$ be a bounded linear operator on~$H$.
The \emph{numerical range} of $T$ is defined by
\[
W(T):=\{\langle Tx,x\rangle: x\in H,~\|x\|=1\}.
\]
It is a convex set
whose closure contains the spectrum of $T$. 
If $\dim H<\infty$, then $W(T)$ is compact.
The \emph{numerical radius} of $T$ is defined by
\[
w(T):=\sup\{|\langle Tx,x\rangle|: x\in H,~\|x\|=1\}.
\]
It is related to the operator norm via 
the   double inequality
\begin{equation}\label{E:nrnorm}
\|T\|/2\le w(T) \le \|T\|.
\end{equation}
If further $T$ is self-adjoint, then $w(T)=\|T\|$.
For proofs of these facts and further background on numerical range 
we refer to the book of Gustafson and Rao \cite{GR97}.

This paper arose from an attempt to gain 
a better understanding of mapping theorems for numerical ranges.
In contrast with spectra, 
it is not true in general that $W(p(T))=p(W(T))$ for polynomials $p$, 
nor is it true if we take convex hulls of both sides. 
However, some partial results do hold. 
Perhaps the most famous of these is the power inequality: 
for all $n\ge1$, we have
\[
w(T^n)\le w(T)^n. 
\]
This was conjectured by Halmos and, after  several  partial results, 
was established by Berger \cite{Be65} using dilation theory. 
An elementary proof was given by Pearcy in \cite{Pe66}. 
A more general result was established by Berger and Stampfli in \cite{BS67}
for functions in the disk algebra
(namely functions that are continuous on the closed unit disk and holomorphic on 
the open unit disk).
They showed that, if $w(T)\le1$, then, for all $f$ in the disk algebra 
such that $f(0)=0$, we have
\[
w(f(T))\le\|f\|_\infty.
\] 
Again their proof  used dilation theory.
In \S\ref{S:Berger} below, we give an elementary proof of this result 
along the lines of Pearcy's proof of the power inequality.

The assumption that $f(0)=0$ is essential in the Berger--Stampfli theorem, 
as is shown by an example in \cite{BS67}.
Without this assumption, the situation becomes more complicated. 
The best result in this setting is Drury's teardrop theorem \cite{Dr08},
which will be discussed in detail in \S\ref{S:Drury} below.
At the heart of the teardrop theorem is an operator inequality, 
which Drury proved by citing a decomposition theorem of Dritschel and Woerdeman \cite{DW97}, 
and then performing some  complicated calculations.
It turns out that these difficulties can be circumvented 
by exploiting a refinement of the inequality \eqref{E:nrnorm}.
We establish this refinement in \S\ref{S:nrineq} 
and show how it can be used to simplify Drury's argument in \S\ref{S:Drury}. 
In \S\ref{S:conclusion} we make some concluding remarks.

\section{An elementary proof of the Berger--Stampfli mapping theorem}\label{S:Berger}

In this section we present an elementary proof 
of the aforementioned theorem of Berger and Stampfli. 
Here is the formal statement of the theorem.

\begin{theorem}\label{T:Berger}
Let $H$ be a complex Hilbert space, 
let $T$ be a bounded linear operator on $H$ with $w(T)\le1$, 
and let $f$ be a function in the disk algebra such that $f(0)=0$. 
Then $w(f(T))\le\|f\|_\infty$.
\end{theorem}

We require two folklore lemmas about finite Blaschke products.
Let us write $\DD$ for the open unit disk and $\TT$ for the unit circle.

\begin{lemma}\label{L:Blaschke1}
Let $B$ be a finite Blaschke product. 
Then $\zeta B'(\zeta)/B(\zeta)$ is real and strictly positive for all $\zeta\in\TT$.
\end{lemma}

\begin{proof}
We can write
\[
B(z)=c\prod_{k=1}^n\frac{a_k-z}{1-\overline{a}_kz},
\]
where $a_1,\dots,a_n\in\DD$ and $c\in\TT$. Then
\[
\frac{B'(z)}{B(z)}
=\sum_{k=1}^n\frac{1-|a_k|^2}{(z-a_k)(1-\overline{a}_kz)}.
\]
In particular, if $\zeta\in\TT$, then
\[
\frac{\zeta B'(\zeta)}{B(\zeta)}
=\sum_{k=1}^n\frac{1-|a_k|^2}{|\zeta-a_k|^2},
\]
which is real and strictly positive.
\end{proof}

\begin{lemma}\label{L:Blaschke2}
Let $B$ be a Blaschke product of degree $n$ such that $B(0)=0$. 
Then, given $\gamma\in\TT$,
there exist $\zeta_1,\dots,\zeta_n\in\TT$
and $c_1,\dots,c_n>0$ such that
\begin{equation}\label{E:Blaschke2}
\frac{1}{1-\overline{\gamma}B(z)}
=\sum_{k=1}^n\frac{c_k}{1-\overline{\zeta}_kz}.
\end{equation}
\end{lemma}

\begin{proof}
Given $\gamma\in\TT$, 
the roots of the equation $B(z)=\gamma$ lie on the unit circle, 
and by Lemma~\ref{L:Blaschke1} they are simple. 
Call them $\zeta_1,\dots,\zeta_n$.
Then $1/(1-\overline{\gamma}B)$ has simple poles at the $\zeta_k$. 
Also, as $B(0)=0$, we have $B(\infty)=\infty$ 
and so $1/(1-\overline{\gamma}B)$ vanishes at $\infty$. 
Expanding it in partial fractions gives \eqref{E:Blaschke2}, 
for some choice of $c_1,\dots,c_n\in\CC$.

The coefficients $c_k$ are easily evaluated. 
Indeed, from \eqref{E:Blaschke2} we have
\[
c_k
=\lim_{z\to\zeta_k}\frac{1-\overline{\zeta}_kz}{1-\overline{\gamma}B(z)}
=\lim_{z\to\zeta_k}\frac{(\zeta_k-z)/\zeta_k}{(B(\zeta_k)-B(z))/B(\zeta_k)}
=\frac{B(\zeta_k)}{\zeta_kB'(\zeta_k)}.
\]
In particular $c_k>0$ by Lemma~\ref{L:Blaschke1}.
\end{proof}

\begin{proof}[Proof of Theorem~\ref{T:Berger}]
Suppose first that $f$ is a finite Blaschke product $B$. 
Suppose also that the spectrum $\sigma(T)$ of $T$ 
lies within the open unit disk $\DD$. 
By the spectral mapping theorem 
$\sigma(B(T))=B(\sigma(T))\subset\DD$ as well.
Let $x\in H$ with $\|x\|=1$. 
Given $\gamma\in\TT$, 
let $\zeta_1,\dots,\zeta_n\in\TT$ and $c_1,\dots,c_n>0$
as in Lemma~\ref{L:Blaschke2}. Then we have
\begin{align*}
1-\overline{\gamma}\langle B(T)x,x\rangle
&=\langle (I-\overline{\gamma}B(T))x,x\rangle\\
&=\langle y,(I-\overline{\gamma}B(T))^{-1}y\rangle 
&&\text{where~}y:=(I-\overline{\gamma}B(T))x\\
&=\Bigl\langle y,\sum_{k=1}^nc_k (I-\overline{\zeta}_kT)^{-1}y\Bigr\rangle
&&\text{by~}\eqref{E:Blaschke2}\\
&=\sum_{k=1}^n c_k\langle(I-\overline{\zeta}_kT)z_k,z_k\rangle 
&&\text{where~}z_k:=(I-\overline{\zeta}_kT)^{-1}y\\
&=\sum_{k=1}^n c_k(\|z_k\|^2-\overline{\zeta}_k\langle Tz_k,z_k\rangle).
\end{align*}
Since $w(T)\le1$, we have 
$\Re(\|z_k\|^2-\overline{\zeta}_k\langle Tz_k,z_k\rangle)\ge0$, 
and as $c_k>0$ for all $k$, it follows that 
\[
\Re(1-\overline{\gamma}\langle B(T)x,x\rangle)\ge0.
\]
As this holds for all $\gamma\in\TT$ and all $x$ of norm $1$, 
it follows that $w(B(T))\le1$.

Next we relax the assumption on $f$, 
still assuming that $\sigma(T)\subset\DD$. 
We can suppose that $\|f\|_\infty=1$. 
Then there exists a sequence of finite Blaschke products $B_n$ 
that converges locally uniformly to $f$ in $\DD$. 
(This is Carath\'eodory's theorem: a simple proof can be found in \cite[\S1.2]{Ga07}.) 
Moreover, as $f(0)=0$, we can also arrange that $B_n(0)=0$ for all $n$.
By what we have proved, $w(B_n(T))\le1$ for all $n$.
Also $B_n(T)$ converges in norm to $f(T)$, because $\sigma(T)\subset\DD$. 
It follows that $w(f(T))\le1$, as required.

Finally we relax the assumption that $\sigma(T)\subset\DD$.
By what we have already proved, 
$w(f(rT))\le\|f\|_\infty$ for all $r<1$. 
Interpreting $f(T)$ as $\lim_{r\to1^-}f(rT)$, 
it follows that  $w(f(T))\le\|f\|_\infty$, 
\emph{provided} that this limit exists. 
In particular this is true when $f$ is holomorphic in a neighborhood of $\overline{\DD}$. 
To prove the existence of the limit in the general case, 
we proceed as follows. 
Given $r,s\in(0,1)$, 
the function $g_{rs}(z):=f(rz)-f(sz)$ 
is holomorphic in a neighborhood of $\overline{\DD}$ 
and vanishes at~$0$, 
so, by what we have already proved, 
$w(g_{rs}(T))\le\|g_{rs}\|_\infty$. Therefore,
\[
\|f(rT)-f(sT)\|=\|g_{rs}(T)\|\le 2w(g_{rs}(T))\le 2\|g_{rs}\|_\infty.
\]
The right-hand side tends to zero as $r,s\to1^{-}$, so,
by the usual Cauchy-sequence argument, 
$f(rT)$ converges as $r\to1^-$. 
This completes the proof.
\end{proof}

\section{A local inequality relating norm to numerical radius}\label{S:nrineq}

Let $T$ be a bounded  operator on a Hilbert space $H$ and let $x\in H$.
The left-hand inequality in \eqref{E:nrnorm} amounts to saying  
that $\|Tx\|\le2$ whenever $w(T)\le 1$ and $\|x\|\le1$.
In this section we establish the following local refinement.

\begin{theorem}\label{T:nrineq}
If $w(T)\le1$ and $\|x\|\le1$, then
\begin{equation}\label{E:nrineq}
\|Tx\|^2\le 2+2\sqrt{1-|\langle Tx,x\rangle|^2}.
\end{equation}
\end{theorem}

\begin{proof}
We may as well suppose that $\|x\|=1$. 
Multiplying $T$ by a unimodular scalar, 
we may further suppose that $\langle Tx,x\rangle\ge0$.
Set $A:=(T+T^*)/2$ and $B:=(T-T^*)/2i$.
By the triangle inequality, we then have
\[
\|Tx-\langle Tx,x\rangle x\|
\le \|Ax-\langle Ax,x\rangle x\|+\|Bx-\langle Bx,x\rangle x\|.
\]
Now, by Pythagoras' theorem,
\[
\|Tx-\langle Tx,x\rangle x\|^2=\|Tx\|^2-|\langle Tx,x\rangle|^2,
\]
and likewise for $A$ and $B$.
Also $A$ and $B$ are self-adjoint operators and have numerical radius at most $1$, 
so $\|A\|\le1$ and $\|B\|\le1$.
Further, the condition $\langle Tx,x\rangle\ge0$ implies that
$\langle Ax,x\rangle=\langle Tx,x\rangle$ and $\langle Bx,x\rangle=0$. Hence
\begin{align*}
\|Ax-\langle Ax,x\rangle x\|^2&=\|Ax\|^2-|\langle Ax,x\rangle|^2\le 1-|\langle Tx,x\rangle|^2\\
\intertext{and}
\|Bx-\langle Bx,x\rangle x\|^2&=\|Bx\|^2-|\langle Bx,x\rangle|^2\le 1.
\end{align*}
Combining all these inequalities, we obtain
\[
\sqrt{\|Tx\|^2-|\langle Tx,x\rangle|^2}
\le \sqrt{1-|\langle Tx,x\rangle|^2}+1,
\]
which, after simplification, yields \eqref{E:nrineq}.
\end{proof}

From Theorem~\ref{T:nrineq} we derive the following operator inequality. 
This result will be needed in the next section.

\begin{theorem}\label{T:operineq}
If $w(T)\le1$, then
\begin{equation}\label{E:operineq}
I+t(T+T^*)+(t^2-1/4)T^*T\ge0 
\qquad(t\in[0,1/2]).
\end{equation}
\end{theorem}

\begin{proof}
The  inequality \eqref{E:operineq} says that, 
for all $x\in H$ with $\|x\|=1$, we have
\[
1+2t\Re\langle Tx,x\rangle+(t^2-1/4)\|Tx\|^2\ge0
\qquad(t\in[0,1/2]).
\]
To prove this, we consider two cases. First, if $\|Tx\|^2\le2$,
then, for all $t\in[0,1/2]$,
\begin{align*}
1+2t\Re\langle Tx,x\rangle+(t^2-1/4)\|Tx\|^2
&\ge 1+2t\Re\langle Tx,x\rangle +2(t^2-1/4)\\
&=2\Bigl|t+\frac{\langle Tx,x\rangle}{2}\Bigr|^2+\frac{1-|\langle Tx,x\rangle|^2}{2}\ge0.
\end{align*}
The other possibility is that $\|Tx\|^2>2$. In this case,
writing \eqref{E:nrineq} in the form 
$\|Tx\|^2-2\le2\sqrt{1-|\langle Tx,x\rangle|^2}$ 
and squaring both sides, we get
\[
4\|Tx\|^2-\|Tx\|^4-4|\langle Tx,x\rangle|^2\ge0.
\]
Then, for all $t\in[0,1/2]$, we have
\begin{align*}
1+&2t\Re\langle Tx,x\rangle+(t^2-1/4)\|Tx\|^2\\
&=\|Tx\|^2\Bigl|t+\frac{\langle Tx,x\rangle}{\|Tx\|^2}\Bigr|^2+
\frac{4\|Tx\|^2-\|Tx\|^4-4|\langle Tx,x\rangle|^2}{4\|Tx\|^2}\ge0.\qedhere
\end{align*}
\end{proof}

\section{Teardrops and Drury's theorem}\label{S:Drury}

If we formulate the Berger--Stampfli theorem as a mapping theorem, 
it says that, whenever $f:\overline{\DD}\to\overline{\DD}$ 
belongs to the disk algebra and satisfies $f(0)=0$, 
we have
\[
W(T)\subset\overline{\DD}\quad\Rightarrow\quad W(f(T))\subset\overline{\DD}.
\]
Without the assumption that $f(0)=0$, this is no longer true.
In this case, the best result is a theorem due to Drury \cite{Dr08}. 
To state his result, we need to introduce some terminology.

Given $\alpha\in\overline{\DD}$, we define the `teardrop region'  
\[
\td(\alpha):=\conv\Bigl(\overline{D}(0,1)\cup\overline{D}(\alpha,1-|\alpha|^2)\Bigr),
\]
namely,  the convex hull of the union of the closed unit disk 
and the closed disk of center $\alpha$ and radius $1-|\alpha|^2$ 
(see Figure~\ref{F:teardrop}).

\begin{figure}[t]
\includegraphics[width=0.5\textwidth]{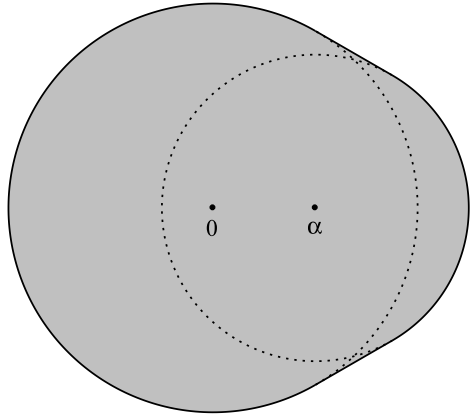}
\caption{teardrop($\alpha$)}\label{F:teardrop}
\end{figure}

Drury's theorem can now be stated as follows.

\begin{theorem}\label{T:Drury}
Let $T$ be an operator on a Hilbert space $H$ such that 
$W(T)\subset\overline{\DD}$, 
and let $f:\overline{\DD}\to\overline{\DD}$ be a function in the disk algebra. 
Then
\[
W(f(T))\subset \td(f(0)).
\] 
\end{theorem}

This has the following immediate consequence.

\begin{corollary}\label{C:Drury}
Under the same hypotheses,
\[
w(f(T))\le 1+|f(0)|-|f(0)|^2\le 5/4.
\]
\end{corollary}

The rationale for these results, which also demonstrates their sharpness, 
is discussed by Drury in \cite{Dr08}. 
Our purpose here is to show how our results in the preceding sections 
fit into the proof of Theorem~\ref{T:Drury}.

Following Drury, we define
\begin{align*}
&Q(T,t,s):=I+t(T+T^*)+sT^*T,\\
&S:=\Bigl\{(t,s)\in\RR^+\times\RR: 
~w(T)\le1~\Rightarrow ~Q(T,t,s)\ge0\Bigr\}.
\end{align*}
In a section entitled `the key issue', 
Drury gives the following  description of $S$.

\begin{theorem}\label{T:S}
The region $S$ is specified by the following inequalities:
\[
\begin{cases}
s\ge t^2-1/4, &\text{if~}0\le t\le1/2,\\
s\ge 2t-1, &\text{if~}1/2\le t\le 1,\\
s\ge t^2, &\text{if~}t\ge1.
\end{cases}
\]
\end{theorem}

A picture of $S$ is given in  Figure~\ref{F:Druryregion}.

\begin{figure}[t]
\includegraphics[width=0.5\textwidth]{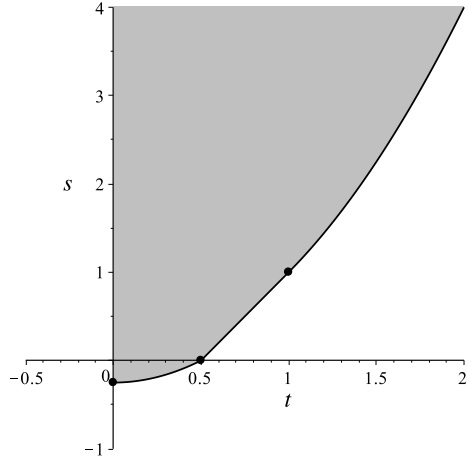}
\caption{The region $S$}\label{F:Druryregion}
\end{figure}

\begin{proof}
We divide the argument into three cases, according to the value of $t$.

\emph{Case 1}: $0\le t\le 1/2$. 
In this case, Theorem~\ref{T:operineq} shows that, 
if $s\ge t^2-1/4$, then, for all $T$ with $w(T)\le1$,  
\[
Q(T,t,s)\ge I+t(T+T^*)+(t^2-1/4)T^*T\ge0.
\] 
On the other hand, if 
$s<t^2-1/4$ and 
$T:=
\begin{pmatrix}
0 &2\\0 &0
\end{pmatrix}$,
then $w(T)\le1$ and
\[
Q(T,t,s)=
\begin{pmatrix}
1 &2t \\ 2t & 1+4s
\end{pmatrix}
\not\ge0,
\]
because it has a negative determinant.
Thus, for this range of values of $t$,
we have $(t,s)\in S\iff s\ge t^2-1/4$.

\emph{Case 2:} $1/2\le t\le 1$.
In this case, if $s\ge2t-1$, then, for all $T$ with $w(T)\le1$,
\begin{align*}
Q(T,t,s)
&\ge I+t(T+T^*)+(2t-1)T^*T\\
&=(1-t)(2I-(T+T^*))+(2t-1)(I+T)^*(I+T)\ge0.
\end{align*}
On the other hand, if $s<2t-1$ and $T:=-I$, then $w(T)\le1$ and
\[
Q(T,t,s)=(1-2t+s)I\not\ge0.
\]
Therefore, for this range of values of $t$,
we have $(t,s)\in S\iff s\ge2t-1$.

\emph{Case 3:} $t\ge1$.
In this case, if $s\ge t^2$, then, for all $T$ with $w(T)\le1$,
\begin{align*}
Q(T,t,s)
&\ge I+t(T+T^*)+t^2T^*T\\
&=(I+tT)^*(I+tT)\ge0.
\end{align*}
On the other hand, if $t\le s<t^2$ and $T:=-(t/s)I$,
then $w(T)\le1$ and
\[
Q(T,t,s)=(1-t^2/s)I\not\ge0.
\]
Thus, for this range of values of $t$,
we have $(t,s)\in S\iff s\ge t^2$.
\end{proof}

\begin{remark}
The main novelty in the proof above 
is the use of Theorem~\ref{T:operineq} in Case~1, 
which shortens the argument considerably.
\end{remark}

\begin{proof}[Proof of Theorem~\ref{T:Drury}]
We follow the method of Drury, with a few details added.

Set $\alpha:=f(0)$. We can suppose that $|\alpha|<1$,
otherwise $f$ is constant and the whole result becomes trivial. 
Let $\phi_\alpha$ be the disk automorphism defined by
\[
\phi_\alpha(z):=\frac{\alpha+z}{1+\overline{\alpha}z},
\]
and set $g:=\phi_\alpha^{-1}\circ f$.
Then $g$ belongs to the disk algebra, $\|g\|_\infty\le1$ and $g(0)=0$. 
By Theorem~\ref{T:Berger} we have $W(g(T))\subset\overline{\DD}$. 
Since $f=\phi_\alpha\circ g$, we may proceed by replacing $T$ by $g(T)$ 
and just studying the case $f=\phi_\alpha$.  
As $\phi_\alpha(T)=\phi_{|\alpha|}(e^{-i\arg\alpha}T)$, 
we may also assume that $\alpha\in[0,1)$. 

Now $\td(\alpha)$ is the intersection of the two families of half-planes
\[
\{z: \Re(e^{-i\theta}z)\le 1\} ~~(\cos\theta\le\alpha)
\quad\text{and}\quad
\{z:\Re(e^{-i\theta}(z-\alpha))\le 1-\alpha^2\}
~~(\cos\theta\ge\alpha).
\]
So, to show $W(\phi_\alpha(T))\subset\td(\alpha)$,
it suffices to prove that
\begin{equation}\label{E:ineq1}
\Re(e^{-i\theta}\phi_\alpha(T))\le I 
\qquad(\cos\theta\le\alpha)
\end{equation}
and 
\begin{equation}\label{E:ineq2}
\Re(e^{-i\theta}(\phi_\alpha(T)-\alpha I))\le (1-\alpha^2)I 
\qquad(\cos\theta\ge\alpha).
\end{equation}

We begin by proving \eqref{E:ineq1}.
This inequality is equivalent to
\[
2I-e^{-i\theta}\phi_\alpha(T)-e^{i\theta}\phi_\alpha(T^*)\ge0.
\]
Given  operators $A,B$ with $B$ invertible,
we have $A\ge0\iff B^*AB\ge0$.
Applying this with $A$ equal to the left-hand side above and $B:=(I+\alpha T)$, 
we see that the desired inequality is equivalent to
\[
2(1-\alpha\cos\theta)I
+(2\alpha-e^{i\theta}-\alpha^2e^{-i\theta})T
+(2\alpha-e^{-i\theta}-\alpha^2e^{i\theta})T^*
+2\alpha(\alpha-\cos\theta)T^*T\ge0.
\]
Set
\[
\omega:=
\frac{2\alpha-e^{i\theta}-\alpha^2e^{-i\theta}}{|2\alpha-e^{i\theta}-\alpha^2e^{-i\theta}|}
=\frac{2\alpha-e^{i\theta}-\alpha^2e^{-i\theta}}{1-2\alpha\cos\theta+\alpha^2}.
\]
Then we may rewrite the last inequality as
\[
2(1-\alpha\cos\theta)I
+(1-2\alpha\cos\theta+\alpha^2)(\omega T+(\omega T)^*)
+2\alpha(\alpha-\cos\theta)(\omega T)^*(\omega T)\ge0,
\]
or equivalently, $Q(\omega T,t,s)\ge0$, where
\[
t:=\frac{1-2\alpha\cos\theta+\alpha^2}{2(1-\alpha\cos\theta)}
\qquad\text{and}\qquad
s:=\frac{\alpha(\alpha-\cos\theta)}{1-\alpha\cos\theta}=2t-1.
\]
It is elementary to verify that, for $-1\le\cos\theta\le\alpha$,  
the parameter $t$ stays in the interval $[1/2,1]$. 
Hence, by Theorem~\ref{T:S}, we do indeed have $Q(\omega T,t,s)\ge0$. 
This establishes \eqref{E:ineq1}.

Now we turn to \eqref{E:ineq2}. 
This inequality is equivalent to
\[
2I-e^{-i\theta}\psi_\alpha(T)-e^{i\theta}\psi_\alpha(T^*)\ge0,
\]
where $\psi_\alpha(z):=z/(1+\alpha z)$.
As before, considering $B^*AB$ with $B:=(I+\alpha T)$,
we see that the preceding inequality is equivalent to
\[ 
2I+(2\alpha-e^{-i\theta})T
+(2\alpha-e^{i\theta}T^*)
+2\alpha(\alpha-\cos\theta)T^*T\ge0.
\]
Set
\[
\omega:=\frac{2\alpha-e^{-i\theta}}{|2\alpha-e^{-i\theta}|}
=\frac{2\alpha-e^{-i\theta}}{2\sqrt{\alpha(\alpha-\cos\theta)+1/4}}.
\]
Then we may rewrite the last inequality as 
\[
I+\sqrt{\alpha(\alpha-\cos\theta)+1/4}\bigl(\omega T+(\omega T)^*\bigr)
+\alpha(\alpha-\cos\theta)(\omega T)^*(\omega T)\ge0,
\]
or equivalently, $Q(\omega T,t,s)\ge0$, where
\[
t:=\sqrt{\alpha(\alpha-\cos\theta)+1/4}
\qquad\text{and}\qquad
s:=\alpha(\alpha-\cos\theta)=t^2-1/4.
\]
It is elementary to verify that, for $\alpha\le\cos\theta\le1$,
the parameter $t$ stays in the interval $[0,1/2]$.
Hence, by Theorem~\ref{T:S}, we do indeed have $Q(\omega T,t,s)\ge0$. 
This establishes \eqref{E:ineq2}, and completes the proof.
\end{proof}

\begin{remark}
The  part of the numerical range of $f(T)$ `sticking out' of the unit disk is governed by the inequality \eqref{E:ineq2},
which corresponds to the slice of $S$ where $0\le t\le1/2$, which is in turn determined by the operator inequality Theorem~\ref{T:operineq}.
\end{remark}

\section{Concluding remarks}\label{S:conclusion}

\subsection{Aleksandrov--Clark measures}
Lemma~\ref{L:Blaschke2} is a special case of a construction of Clark 
later generalized by Aleksandrov.
Let $f:\DD\to\DD$ be holomorphic with $f(0)=0$.
Then, given  $\gamma \in\TT$, 
there exists a probability measure $\mu_{\gamma}$ on  $\TT$ such that
\begin{equation}\label{E:AC}
\frac{1}{1-\overline{\gamma}f(z)} 
= \int_{\TT} \frac{d\mu_{\gamma}(\zeta)}{1-\overline{\zeta}z}
\qquad(z\in\DD).
\end{equation}
The measures $\mu_\gamma$ are known as Aleksandrov--Clark measures.
For details of their construction and an account of their properties, 
see for example \cite{PS06} and \cite{Sa07}.

It is possible to base a proof of Theorem~\ref{T:Berger}  
directly on the formula \eqref{E:AC}.  
Indeed, a proof very similar to this appears in the paper of Kato \cite{Ka65}.

\subsection{Reformulations of the inequality \eqref{E:nrineq}}

The inequality \eqref{E:nrineq} can be reformulated in various equivalent ways. 
We record two of them here. 

\begin{proposition}
Let $T$ be an operator on a Hilbert space $H$ and let $x\in H$.
If $w(T)\le1$ and $\|x\|\le1$, then
\begin{equation}\label{E:|sin|}
\|Tx\|\le \max\Bigl\{2|\sin\theta|,\sqrt{2}\Bigr\},
\end{equation}
where $\theta$ is the hermitian angle between $x$ and $Tx$.
\end{proposition}

\begin{proof}
By definition of hermitian angle, 
$|\langle Tx,x\rangle|=\|Tx\|\|x\|\cos\theta$.
We can suppose that $\|x\|=1$.
Substituting into \eqref{E:nrineq}, we get
$\|Tx\|^2-2\le 2\sqrt{1-\|Tx\|^2\cos^2\theta}$.
If $\|Tx\|^2\ge 2$, 
then we may square both sides to obtain
\[
\|Tx\|^4-4\|Tx\|^2+4\le 4-4\|Tx\|^2\cos^2\theta,
\]
which leads to $\|Tx\|\le 2|\sin\theta|$. 
On the other hand, if $\|Tx\|^2<2$, 
then obviously $\|Tx\|< \sqrt{2}$. 
Either way, \eqref{E:|sin|} holds.
\end{proof}

\begin{proposition}
If the  matrix
$\begin{pmatrix} a &b\\c &d\end{pmatrix}$
has numerical radius at most $1$,
then
\[
|c|\le 1+\sqrt{1-|a|^2}.
\]
\end{proposition}

\begin{proof}
Applying Theorem~\ref{T:nrineq} with $H=\CC^2$ and $x=(1,0)$,
we obtain
\[
|a|^2+|c|^2\le 2+2\sqrt{1-|a|^2}.
\]
After simplification, this gives the result.
\end{proof}

We mention in passing that there are complete characterizations of operators $T$ such that $w(T)\le1$;
see Ando \cite{An73}.

\subsection{Extension to general domains}

The papers \cite{BS67} and \cite{Ka65} contain some partial extensions of
Theorem~\ref{T:Berger} to certain domains other than the disk. 

More recently, Crouzeix~\cite{Cr07} has shown that, 
if $T$ is any Hilbert-space operator and
$f$ is holomorphic on a neighborhood of $\overline{W(T)}$, then
\begin{equation}\label{E:Crouzeix}
\|f(T)\|\le C\sup_{z\in W(T)}|f(z)|,
\end{equation}
where $C$ is an absolute constant satisfying $C\le 11.08$.
It is conjectured that \eqref{E:Crouzeix} holds with $C=2$.
This is  best possible, as can be seen by considering the matrix
\begin{equation}\label{E:Texample}
T=
\begin{pmatrix}
0 &2\\0 &0 
\end{pmatrix},
\end{equation}
which satisfies $W(T)=\overline{\DD}$ and $\|T\|=2$.

Of course inequality \eqref{E:Crouzeix} implies the numerical-range mapping inequality
\begin{equation}\label{E:Crouzeix2}
w(f(T))\le C\sup_{z\in W(T)}|f(z)|.
\end{equation}
However, in the light of Corollary~\ref{C:Drury},
it is conceivable that the best constant $C$ in \eqref{E:Crouzeix2} 
is actually smaller than $2$.
The best that we can hope for is $C=5/4$.
Indeed, taking $T$ as in \eqref{E:Texample} and $f(z):=(1-2z)/(2-z)$,
we have 
\[
\sup_{z\in W(T)}|f(z)|=\sup_{z\in\overline{\DD}}\Bigl|\frac{1-2z}{2-z}\Bigr|=1,
\]
while $f(T)=I/2-3T/4$, 
which has numerical range $\overline{D}(1/2,3/4)$, 
so $w(f(T))=5/4$.

\bibliographystyle{amsplain}

\end{document}